\documentclass{article}
 \usepackage[margin=1in]{geometry} 
\usepackage{amsmath,amsthm,amssymb,amsfonts, fancyhdr, comment, parskip, tikz, caption, subcaption, graphicx}
\usepackage[colorlinks]{hyperref}
\usetikzlibrary{graphs}
\usetikzlibrary{shadings}
\usetikzlibrary{shapes.geometric}
\usetikzlibrary{arrows}

\definecolor{purple}{RGB}{148,0,211}
\definecolor{green}{RGB}{50,205,50}

\newtheorem{theorem}{Theorem}[section]

\newtheorem{lemma}[theorem]{Lemma}
\newtheorem{definition}[theorem]{Definition}
\newcommand{\naturals}{\mathbb{N}}
\newcommand{\reals}{\mathbb{R}}
\newcommand{\TT}{{\cal T}}
\newcommand{\CC}{{\cal C}}
\newcommand{\FF}{{\cal F}}

\newcommand{\HH}{{\cal H}}
\newcommand{\DD}{{\cal D}}

\newcommand{\PP}{{\cal P}}

\newcommand{\SSS}{{\cal S}}
\newcommand{\KK}{{\cal K}}

\newcommand{\MM}{{\cal M}}
\newcommand{\WW}{{\cal W}}
\newcommand{\EE}{{\cal E}}
\newcommand{\prob}{\mathbb{P}}
\newcommand{\eps}{{\varepsilon}}
\tikzstyle{arch} = [out=30, in=150]

\title{Edge-coloring a graph $G$ so that every copy of a graph $H$ has an odd color class}
\author{Patrick Bennett \footnote{Department of Mathematics, Western Michigan University, Kalamazoo, MI. \texttt{patrick.bennett@wmich.edu}.}  
\and Emily Heath \footnote {Department of Mathematics, Iowa State University, Ames, IA. \texttt{eheath@iastate.edu}. Research partially supported by NSF RTG Grant DMS-1839918.}
\and Shira Zerbib \footnote {Department of of Mathematics, Iowa State University, Ames, IA. \texttt{zerbib@iastate.edu}. Research partially supported by NSF Grant DMS-1953929.}}
\date{\today}

\begin{document}

\maketitle

\begin{abstract}
Recently,  Alon \cite{alon} introduced the  notion of an $H$-code for a graph $H$:
a collection of graphs  on vertex set  $[n]$ is  
an $H$-code 
if it contains no two members whose symmetric difference is isomorphic to
$H$. 
Let $D_{H}(n)$ denote the maximum possible cardinality of an $H$-code, and let
$d_{H}(n)=D_{H}(n)/2^{n \choose 2}$.
Alon observed that a  lower bound on $
d_{H}(n)$ can be obtained by attaining an upper bound on the number of colors needed to 
edge-color $K_n$ so that every copy of $H$ has an odd color class.

Motivated by this observation, we define  $g(G,H)$ to be the minimum number of colors needed to edge-color a graph $G$ so that every copy of  $H$ has an odd color class. We prove $g(K_n,K_5) \le n^{o(1)}$ and $g(K_{n,n}, C_4)= n/2+o(n)$. The first result shows $d_{K_5}(n) \ge \frac{1}{n^{o(1)}}$ and was obtained independently in \cite{GXZ}.   
\end{abstract}

\section{Introduction}

Given a graph $G$ and a subgraph $H$, let $g(G,H)$ be the minimum number of colors needed to edge-color $G$ so that every copy of $H$ sees some color an odd number of times. The problem of determining $g(G,H)$ is a natural question motivated by recent work of Alon~\cite{alon} introducing the related notion of {\em graph-codes}.

 Let $V=[n]$ and let
$\HH$ be a family of graphs on the set of vertices $[n]$ 
which is closed under
isomorphism.
A collection of graphs $\FF$ on $[n]$ is called 
an {\em $\HH$-code} if it contains no two members whose symmetric difference is a graph in $\HH$. For the special case that $\HH$ contains all copies 
of a single graph $H$ on $[n]$ this is called an $H$-code. 
Let $D_{\HH}(n)$ denote the maximum possible cardinality of an $\HH$-code, and let
\[
d_{\HH}(n)=\frac{D_{\HH}(n)}{2^{n \choose 2}}
\]
be the maximum possible fraction of the total number of graphs
on $[n]$ in an $\HH$-code.  If $\HH$ consists of all graphs
isomorphic to one graph $H$,  denote $d_{\HH}(n)$ by $d_{H}(n)$.

In the case where $\HH$ consists of all graphs with independence number at most 2, a result of Ellis, Filmus
and Friedgut~\cite{EFF} shows that $d_{\HH}(n)=1/8$ for $n\geq 3$. Berger and Zhao~\cite{BZ} proved the analogous result for the family $\HH$ of graphs with independence number at most 3, showing that $d_{\HH}(n)=1/64$ for all $n\geq 4$. In~\cite{AGKMS}, Alon, Gujgiczer, Körner, Milojević, and Simonyi studied $D_{\HH}(n)$ and $d_{\HH}(n)$ for several families $\HH$, such as disconnected graphs, graphs that are not 2-connected, non-Hamiltonian graphs, graphs that contain or do not contain a spanning star, graphs that contain an induced or non-induced copy of a fixed graph $T$, and graphs that do not contain such a subgraph. Alon~\cite{alon} also studied the cases for cliques, stars, and matchings.

In~\cite{alon}, Alon mentions that the case $\HH=\KK$, where $\KK$ is the family of all cliques, is of particular interest. This case is motivated by a conjecture of Gowers (see \cite{Go, alon}). Further, he comments that $d_{K_4}(n)\geq \frac{1}{n^{o(1)}}.$ This result follows from the existence of an edge-coloring of $K_n$ by $n^{o(1)}$ colors with no copy of $K_4$ in which every color appears an even number of times; that is, $g(K_n,K_4)\leq n^{o(1)}$. A coloring with this property was given in \cite{CH2}, modifying constructions in~\cite{CFLS, Mubayi1}. 
In our first result,  we show that the same coloring has the property that there exists no copy of $K_5$ in which every color appears an even number of times. We prove:
\begin{theorem}\label{main1}
   We have $g(K_n,K_5) \le n^{o(1)}$. 
\end{theorem}

This, together with a similar argument to that in \cite{alon}, implies the following.
\begin{theorem}\label{mainthm} We have
    $d_{K_5}(n)\geq \frac{1}{n^{o(1)}}.$ 
\end{theorem}

Very recently, Ge, Xu, and Zhang~\cite{GXZ} also independently obtained this result using a similar method.  They further note that $g(K_n,K_5)\geq \Omega(\log n)$. 

As is observed in \cite{alon}, if every member of $\HH$ has an odd number of edges then $d_{\HH}(n) \geq \frac{1}{2}$, as the family of all graphs on $[n]$ with an even number of edges forms an $\HH$-code.
Thus, when considering cliques, the next interesting case is the case of $H=K_8$. We do not know if the same coloring, or some generalization of it, has the property that  there exists no copy of $K_8$ in which every color appears an even number of times.  As our method relies on case analysis, it will be hard to generalize it to the case of $K_8$.

An upper bound on $g(K_n,K_p)$ for larger $p$ follows from recent work of Bennett, Delcourt, Li, and Postle~\cite{BDLP} on the \emph{generalized Ramsey number} $f(n,p,q)$, that is, the minimum number of colors needed to color $E(K_n)$ so that every copy of $K_p$ sees at least $q$ colors. Until recently, the best upper bound on $f(n,p,q)$ for general $p,q$ was the original bound of Erd\H{o}s and Gy\'arf\'as~\cite{EG} obtained using the Lov\'asz Local Lemma. Bennett, Dudek, and English~\cite{BDE} used a random greedy process to improve this bound by a logarithmic factor for values of $q$ and $p$ with $q\leq (p^2-26p+55)/4$. In~\cite{BDLP}, their result is extended to all values of $p$ and $q$ except at the values $q=\binom{p}{2}-p+2$ and $q=\binom{p}{2}-\left\lfloor\frac{p}{2}\right\rfloor+2$, where the local lemma bound is known to be tight. 

\begin{theorem}[Bennett, Delcourt, Li, Postle~\cite{BDLP}]\label{thm:fnpq} 
For fixed positive integers $p,q$ with $p-2$ not divisible by $\binom{p}{2}-q+1$, we have 
\[f(n,p,q)=O\left(\left(\frac{n^{p-2}}{\log n}\right)^{\frac{1}{\binom{p}{2}-q+1}}\right).\]
\end{theorem}

In fact, Theorem~\ref{thm:fnpq} is generalized in~\cite{BDLP} to give an analogous upper bound in a list-coloring setting for any subgraph (not only cliques) and in hypergraphs with higher uniformity. Note that Theorem~\ref{thm:fnpq} with $q=\binom{p}{2}/2+1$ implies that \[g(K_n,K_p)=O\left(\left(\frac{n^{p-2}}{\log n}\right)^{2/\binom{p}{2}}\right).\] Indeed, any $(p,\frac{1}{2}\binom{p}{2}+1)$-coloring of $K_n$ has the stronger condition that some color must appear exactly once in each $p$-clique. 

Our second result focuses on the question of determining $g(K_{n,n},C_4)$. In order to prove an upper bound, we will construct a coloring using the ``forbidden submatching  method" recently introduced by Delcourt and Postle~\cite{DP} and independently by Glock, Joos, Kim, K\"uhn, and Lichev~\cite{GJKKL} as ``conflict-free hypergraph matchings." In~\cite{JM}, Joos and Mubayi use the variant of the method from~\cite{GJKKL} along with a probabilistic argument to show that $f(K_{n,n},C_4,3)=\frac{2}{3}n+o(n)$. Note that this implies $g(K_{n,n},C_4)\leq\frac{2}{3}n+o(n)$.
We use a similar approach to show the stronger upper bound of $g(K_{n,n},C_4)\leq\frac{1}{2}n+o(n)$. As such, combined with a short proof that $g(K_{n,n},C_4)>n/2$, we obtain the following result. 

\begin{theorem}\label{thm:C4inKnn}
We have $g(K_{n,n},C_4)=\frac{1}{2}n+o(n).$
\end{theorem}

The rest of the paper is organized as follows. In Section~\ref{sec:K5}, we prove Theorem~\ref{mainthm}. In Section~\ref{sec:blackbox}, we give the necessary preliminaries to state the main tool for the proof of Theorem~\ref{thm:C4inKnn}, namely the theorem on forbidden submatchings given in~\cite{GJKKL}. Finally, in Section~\ref{sec:C4inKnn}, we prove Theorem~\ref{thm:C4inKnn}.

\section{Proof of Theorem~\ref{main1}}\label{sec:K5}
A \emph{$(p,q)$-coloring} of $K_n$ is an edge-coloring in which each $p$-clique receives at least $q$ colors. In 2015, Conlon, Fox, Lee, and Sudakov~\cite{CFLS} constructed a $(p,p-1)$-coloring with $n^{o(1)}$ colors. In 2017, Cameron and Heath~\cite{CH1} used a modified version of their $(5,4)$-coloring as part of their construction of a $(5,5)$-coloring with $n^{1/3+o(1)}$ colors. We will prove that this ``Modified CFLS" $(5,4)$-coloring using $n^{o(1)}$ colors guarantees that every $K_5$ contains some color an odd number of times.

Let $n=2^{m^2}$ for some positive integers $m$. We will view the vertices of $K_n$ as binary strings of length $m^2$; that is, $V=\{0,1\}^{m^2}$. For each $v\in V$, let $v^{(i)}$ denote the $i$th block of bits of length $m$ in $v$, so \[v=\left(v^{(1)}, v^{(2)}, \ldots, v^{(m)}\right),\] where each $v^{(i)}\in \{0,1\}^{m}$. 

Note that we can assign a linear order to the vertices by considering each to be an integer represented in binary and taking the standard ordering of these integers. That is, $x<y$ if and only if the first bit of difference between $x$ and $y$ is zero in $x$ and one in $y$. Furthermore, each $m$-block of our vertices can be viewed as a binary representation of an integer from 0 to $2^{m}-1$, so the $m$-blocks can be ordered in the same way. 

Let $x,y\in V$ such that $x<y$. Let $i$ be the first index for which $x^{(i)}\neq y^{(i)}$, and for each $k\in[m]$, let $i_k$ be the first index at which a bit of $x^{(k)}$ differs from the corresponding bit of $y^{(k)}$, or $i_k=0$ if $x^{(k)}=y^{(k)}$. In addition, for each $k\in [m]$, let \[\delta_k=\begin{cases} +1 & x^{(k)}\leq y^{(k)}, \\ -1 & x^{(k)}>y^{(k)}.\end{cases}\]
The \emph{Modified CFLS coloring} assigns to the edge $xy$ the color
\[\varphi(xy)=\left(\left(i,\{x,y\}\right),i_1,i_2,\ldots,i_{m},\delta_1,\delta_2,\ldots,\delta_{m}\right).\]

It was shown in \cite{CH1} that the number of colors in the  Modified CFLS coloring is $2^{O(\sqrt{\log n}\log\log n)}.$
We will now show that if $K_n$ is edge-colored by 
the Modified CFLS coloring then every copy of $K_5$ has an odd color class.

Let $E(K_n)$ be colored by the Modified CFLS coloring $\varphi$ given by Cameron and Heath~\cite{CH1}.
Then every copy of $K_5$ contains at least 4 distinct colors.
If a copy of $K_5$ contains at least six colors, then clearly some color appears in this $K_5$ exactly once, giving us the odd color we desire. Therefore, we only need to consider copies of $K_5$ which contain four or five colors under $\varphi$. 
The case where $K_5$ contains exactly four colors is easy to understand; indeed, Cameron and Heath~\cite{CH2} characterized the colorings of $K_5$ with exactly four colors which can appear under $\varphi$.

\begin{theorem} Let $p\geq 3$ and let the edges of $K_n$ be colored with the Modified CFLS $(p,p-1)$-coloring. The only $p$-cliques that contain exactly $p-1$ distinct edge-colors are isomorphic (as edge-colored graphs) to one of the edge-colored $p$-cliques given in the definition below.
\end{theorem}

\begin{definition}
Given an edge-coloring $f:E(K_n)\rightarrow C$, we say that a subset $S \subseteq V(K_n)$ has a \emph{leftover structure under $f$} if either $|S| = 1$ or there exists a bipartition of $S$ into nonempty sets $A$ and $B$ for which
\begin{itemize}
    \item $A$ and $B$ each have a leftover structure under $f$;
    \item $f(A)\cap f(B)=\emptyset$; and 
    \item there is a fixed color $\alpha\in C$ such that $f(a,b)=\alpha$ for all $a\in A$ and all $b\in B$, and $\alpha\notin f(A)$ and $\alpha\notin f(B)$.
\end{itemize}
\end{definition}

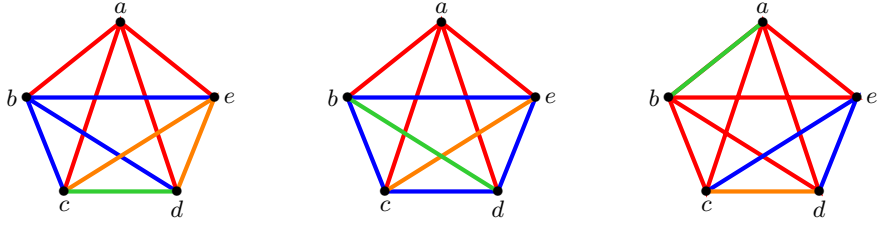
\begin{figure}
\begin{center}
\begin{tikzpicture}[scale=1]
     \coordinate (a) at (1.75,2.75);
     \coordinate (b) at (.5,1.75);
     \coordinate (c) at (1,.5);
     \coordinate (d) at (2.5,.5);
     \coordinate (e) at (3,1.75);

     \draw[ultra thick, red] (e)--(a)--(b) (c)--(a)--(d);
     \draw[ultra thick, blue] (c)--(b)--(e) (b)--(d);
     \draw[ultra thick, orange] (c)--(e)--(d);
     \draw[ultra thick, green] (c)--(d);

     \node[black] at (a) {\small $\bullet$};
     \node[black, above] at (a) {\small $a$};
     \node[black] at (b) {\small $\bullet$};
     \node[black, left] at (b) {\small $b$};
     \node[black] at (c) {\small $\bullet$};
     \node[black, below] at (c) {\small $c$};
     \node[black] at (d) {\small $\bullet$};
     \node[black, below] at (d) {\small $d$};
     \node[black] at (e) {\small $\bullet$};
     \node[black, right] at (e) {\small $e$};
\end{tikzpicture}
\hspace{.3in}
\begin{tikzpicture}[scale=1]
     \coordinate (a) at (1.75,2.75);
     \coordinate (b) at (.5,1.75);
     \coordinate (c) at (1,.5);
     \coordinate (d) at (2.5,.5);
     \coordinate (e) at (3,1.75);

     \draw[ultra thick, red] (e)--(a)--(b) (c)--(a)--(d);
     \draw[ultra thick, blue] (b)--(e)--(d)--(c)--(b);
     \draw[ultra thick, orange] (c)--(e);
     \draw[ultra thick, green] (b)--(d);

     \node[black] at (a) {\small $\bullet$};
     \node[black, above] at (a) {\small $a$};
     \node[black] at (b) {\small $\bullet$};
     \node[black, left] at (b) {\small $b$};
     \node[black] at (c) {\small $\bullet$};
     \node[black, below] at (c) {\small $c$};
     \node[black] at (d) {\small $\bullet$};
     \node[black, below] at (d) {\small $d$};
     \node[black] at (e) {\small $\bullet$};
     \node[black, right] at (e) {\small $e$};
\end{tikzpicture}
\hspace{.3in}
\begin{tikzpicture}[scale=1]
     \coordinate (a) at (1.75,2.75);
     \coordinate (b) at (.5,1.75);
     \coordinate (c) at (1,.5);
     \coordinate (d) at (2.5,.5);
     \coordinate (e) at (3,1.75);

     \draw[ultra thick, red] (b)--(e)--(a)--(b)--(c)--(a)--(d)--(b);
     \draw[ultra thick, blue] (c)--(e)--(d);
     \draw[ultra thick, green] (a)--(b);
     \draw[ultra thick, orange] (c)--(d);

     \node[black] at (a) {\small $\bullet$};
     \node[black, above] at (a) {\small $a$};
     \node[black] at (b) {\small $\bullet$};
     \node[black, left] at (b) {\small $b$};
     \node[black] at (c) {\small $\bullet$};
     \node[black, below] at (c) {\small $c$};
     \node[black] at (d) {\small $\bullet$};
     \node[black, below] at (d) {\small $d$};
     \node[black] at (e) {\small $\bullet$};
     \node[black, right] at (e) {\small $e$};
\end{tikzpicture}
\caption{The three leftover colorings of $K_5$ with four colors under $\varphi$}
\label{fig:leftover K5}
\end{center}
\end{figure}

In particular, there are only three such leftover configurations in the $p=5$ setting, shown in Figure~\ref{fig:leftover K5} below. Note that in each case, there is a color which appears exactly once. Thus, it suffices to consider colorings of $K_5$ with exactly five colors. 

Assume towards a contradiction that $\varphi$ permits some coloring of $K_5$ with exactly 5 colors, each appearing an even number of times. Call these colors $\alpha, \beta, \gamma, \pi,$ and $\theta$. Then each color class is either a matching of two disjoint edges or a path with two adjacent edges. We say that such  colorings of $K_5$ are 2-2-2-2-2 colorings.

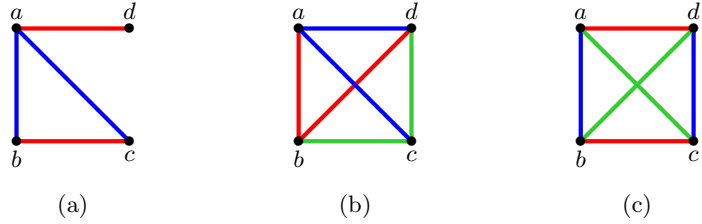
\begin{figure}
\begin{center}
\begin{subfigure}[b]{0.22\textwidth}
\begin{center}
\begin{tikzpicture}[scale=1]
     \coordinate (a) at (0,1.5);
     \coordinate (b) at (0,0);
     \coordinate (c) at (1.5,0);
     \coordinate (d) at (1.5,1.5);

     \draw[ultra thick, red] (a)--(d) (c)--(b);
     \draw[ultra thick, blue] (c)--(a)--(b);

     \node[black] at (a) {\small $\bullet$};
     \node[black, above] at (a) {\small $a$};
     \node[black] at (b) {\small $\bullet$};
     \node[black, below] at (b) {\small $b$};     
     \node[black] at (c) {\small $\bullet$};
     \node[black, below] at (c) {\small $c$};
     \node[black] at (d) {\small $\bullet$};
     \node[black, above] at (d) {\small $d$};
    \end{tikzpicture}
    \caption{}
          \label{fig:forb1}
          \end{center}
     \end{subfigure}
\begin{subfigure}[b]{0.22\textwidth}
\begin{center}
\begin{tikzpicture}[scale=1]
     \coordinate (a) at (0,1.5);
     \coordinate (b) at (0,0);
     \coordinate (c) at (1.5,0);
     \coordinate (d) at (1.5,1.5);

     \draw[ultra thick, red] (a)--(b)--(d);
     \draw[ultra thick, blue] (d)--(a)--(c);
     \draw[ultra thick, green] (d)--(c)--(b);

     \node[black] at (a) {\small $\bullet$};
     \node[black, above] at (a) {\small $a$};
     \node[black] at (b) {\small $\bullet$};
     \node[black, below] at (b) {\small $b$};     
     \node[black] at (c) {\small $\bullet$};
     \node[black, below] at (c) {\small $c$};
     \node[black] at (d) {\small $\bullet$};
     \node[black, above] at (d) {\small $d$};
    \end{tikzpicture}
    \caption{}
          \label{fig:forb2}
          \end{center}
     \end{subfigure}
\begin{subfigure}[b]{0.22\textwidth}
\begin{center}
\begin{tikzpicture}[scale=1]
     \coordinate (a) at (0,1.5);
     \coordinate (b) at (0,0);
     \coordinate (c) at (1.5,0);
     \coordinate (d) at (1.5,1.5);

     \draw[ultra thick, red] (c)--(b) (d)--(a) ;
     \draw[ultra thick, blue] (a)--(b) (c)--(d);
     \draw[ultra thick, green] (a)--(c) (b)--(d);

     \node[black] at (a) {\small $\bullet$};
     \node[black, above] at (a) {\small $a$};
     \node[black] at (b) {\small $\bullet$};
     \node[black, below] at (b) {\small $b$};     
     \node[black] at (c) {\small $\bullet$};
     \node[black, below] at (c) {\small $c$};
     \node[black] at (d) {\small $\bullet$};
     \node[black, above] at (d) {\small $d$};
    \end{tikzpicture}
    \caption{}
          \label{fig:forb3}
     \end{center}
     \end{subfigure}
\caption{Forbidden configurations under $\varphi$ with four vertices}
\label{fig:forbidden 4vxs}
\end{center}
\end{figure}

\begin{lemma}\label{neighborhood3}
   No vertex in a 2-2-2-2-2 coloring of $K_5$ can be incident to edges of only two colors. 
\end{lemma}
\begin{proof}
    If there was such a vertex, then the other four vertices would form a copy of $K_4$ with three colors appearing twice each. But, each such coloring of $K_4$ is one of the configurations shown in Figure~\ref{fig:forbidden 4vxs}, which were shown to be forbidden under $\varphi$ in~\cite{CH1}. 
\end{proof}

Now we will prove a series of lemmas showing that several other common configurations (in Figure~\ref{fig:forbidden 5vxs}) are forbidden under $\varphi$. 
Throughout these proofs, when referring to some color $\alpha$, we will let $\alpha_0$ denote the 0-coordinate of the color (of the form $(i,\{x,y\})$) and let $\alpha_k$ denote the $k$-coordinate of the color, that is, the index of the first bit of difference between the $k$th blocks. Furthermore, 
if $x,y \in \{0,1\}^{m^2}$ are two vectors
and $x^{(k)}=y^{(k)}$ (where $x^{(k)}$ is the $k$th $m$-block of $x$), then we  say that $x$ and $y$ \emph{agree at $k$}; otherwise, we say $x$ and $y$ \emph{disagree at $k$}.

\begin{figure}
\begin{center}
\begin{subfigure}[b]{0.22\textwidth}
\begin{center}
\begin{tikzpicture}[scale=1]
     \coordinate (a) at (1.75,2.75);
     \coordinate (b) at (.5,1.75);
     \coordinate (c) at (1,.5);
     \coordinate (d) at (2.5,.5);
     \coordinate (e) at (3,1.75);

     \draw[ultra thick, blue] (b)--(a) (e)--(d);
     \draw[ultra thick, green] (a)--(e) (b)--(c);
     \draw[ultra thick, red] (c)--(d) (b)--(e);

     \node[black] at (a) {\small $\bullet$};
     \node[black, above] at (a) {\small $a$};
     \node[black] at (b) {\small $\bullet$};
     \node[black, left] at (b) {\small $b$};
     \node[black] at (c) {\small $\bullet$};
     \node[black, below] at (c) {\small $c$};
     \node[black] at (d) {\small $\bullet$};
     \node[black, below] at (d) {\small $d$};
     \node[black] at (e) {\small $\bullet$};
     \node[black, right] at (e) {\small $e$};
\end{tikzpicture}
    \caption{}
          \label{fig:A}
     \end{center}\end{subfigure}
\begin{subfigure}[b]{0.22\textwidth}
\begin{center}
\begin{tikzpicture}[scale=1]
     \coordinate (a) at (1.75,2.75);
     \coordinate (b) at (.5,1.75);
     \coordinate (c) at (1,.5);
     \coordinate (d) at (2.5,.5);
     \coordinate (e) at (3,1.75);

     \draw[ultra thick, red] (b)--(c)--(a);
     \draw[ultra thick, green] (e)--(d)--(a);
     \draw[ultra thick, orange] (a)--(e)--(c);
     \draw[ultra thick, blue] (a)--(b)--(d);

     \node[black] at (a) {\small $\bullet$};
     \node[black, above] at (a) {\small $a$};
     \node[black] at (b) {\small $\bullet$};
     \node[black, left] at (b) {\small $b$};
     \node[black] at (c) {\small $\bullet$};
     \node[black, below] at (c) {\small $c$};
     \node[black] at (d) {\small $\bullet$};
     \node[black, below] at (d) {\small $d$};
     \node[black] at (e) {\small $\bullet$};
     \node[black, right] at (e) {\small $e$};
    \end{tikzpicture}
    \caption{}
          \label{fig:B}
     \end{center}\end{subfigure}
\begin{subfigure}[b]{0.22\textwidth}
\begin{center}
\begin{tikzpicture}[scale=1]
     \coordinate (a) at (1.75,2.75);
     \coordinate (b) at (.5,1.75);
     \coordinate (c) at (1,.5);
     \coordinate (d) at (2.5,.5);
     \coordinate (e) at (3,1.75);

     \draw[ultra thick, red] (b)--(e) (d)--(c);
     \draw[ultra thick, blue] (b)--(a);
     \draw[ultra thick, green] (b)--(c);
     \draw[ultra thick, orange] (a)--(e);
     \draw[ultra thick, purple] (d)--(e);

     \node[black] at (a) {\small $\bullet$};
     \node[black, above] at (a) {\small $a$};
     \node[black] at (b) {\small $\bullet$};
     \node[black, left] at (b) {\small $b$};
     \node[black] at (c) {\small $\bullet$};
     \node[black, below] at (c) {\small $c$};
     \node[black] at (d) {\small $\bullet$};
     \node[black, below] at (d) {\small $d$};
     \node[black] at (e) {\small $\bullet$};
     \node[black, right] at (e) {\small $e$};
    \end{tikzpicture}
    \caption{}
          \label{fig:C}
     \end{center}\end{subfigure}  
     
\begin{subfigure}[b]{0.22\textwidth}
\begin{center}
\begin{tikzpicture}[scale=1]
     \coordinate (a) at (1.75,2.75);
     \coordinate (b) at (.5,1.75);
     \coordinate (c) at (1,.5);
     \coordinate (d) at (2.5,.5);
     \coordinate (e) at (3,1.75);

     \draw[ultra thick, green] (b)--(a)--(c);
     \draw[ultra thick, red] (b)--(c) (d)--(e);
     \draw[ultra thick, blue] (c)--(d)--(a);

     \node[black] at (a) {\small $\bullet$};
     \node[black, above] at (a) {\small $a$};
     \node[black] at (b) {\small $\bullet$};
     \node[black, left] at (b) {\small $b$};
     \node[black] at (c) {\small $\bullet$};
     \node[black, below] at (c) {\small $c$};
     \node[black] at (d) {\small $\bullet$};
     \node[black, below] at (d) {\small $d$};
     \node[black] at (e) {\small $\bullet$};
     \node[black, right] at (e) {\small $e$};
    \end{tikzpicture}
    \caption{}
          \label{fig:D}
     \end{center}\end{subfigure}
\begin{subfigure}[b]{0.22\textwidth}
\begin{center}
\begin{tikzpicture}[scale=1]
     \coordinate (a) at (1.75,2.75);
     \coordinate (b) at (.5,1.75);
     \coordinate (c) at (1,.5);
     \coordinate (d) at (2.5,.5);
     \coordinate (e) at (3,1.75);

     \draw[ultra thick, red] (a)--(b)--(e);
     \draw[ultra thick, blue] (b)--(c)--(a);
     \draw[ultra thick, green] (c)--(d)--(b);
     \draw[ultra thick, orange] (d)--(e)--(c);
     \draw[ultra thick, purple] (e)--(a)--(d);
     
     \node[black] at (a) {\small $\bullet$};
     \node[black, above] at (a) {\small $a$};
     \node[black] at (b) {\small $\bullet$};
     \node[black, left] at (b) {\small $b$};
     \node[black] at (c) {\small $\bullet$};
     \node[black, below] at (c) {\small $c$};
     \node[black] at (d) {\small $\bullet$};
     \node[black, below] at (d) {\small $d$};
     \node[black] at (e) {\small $\bullet$};
     \node[black, right] at (e) {\small $e$};
    \end{tikzpicture}
    \caption{}
          \label{fig:E}
     \end{center}\end{subfigure}
\caption{Forbidden configurations under $\varphi$ with five vertices}
\label{fig:forbidden 5vxs}
\end{center}
\end{figure}
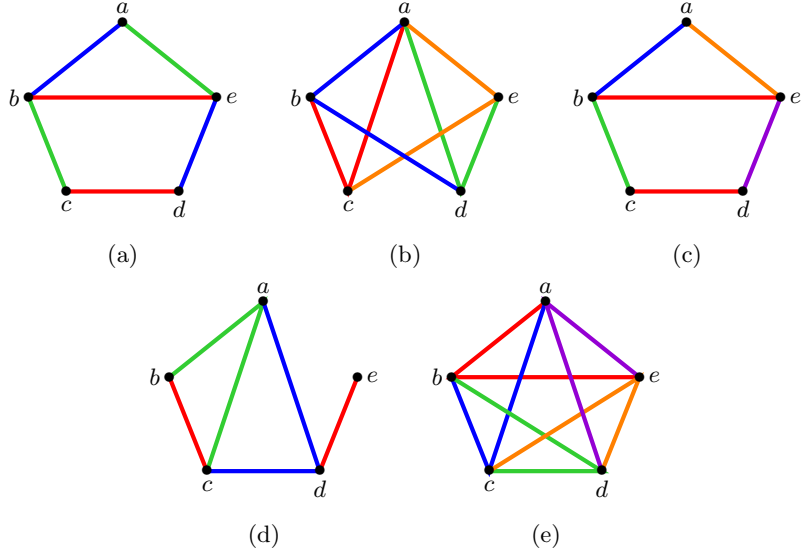

\begin{lemma}\label{lem:candy corn}
The configuration with five distinct vertices $a,b,c,d,e\in V$ for which $\varphi(be)=\varphi(cd)$, $\varphi(ab)=\varphi(de)$, and $\varphi(bc)=\varphi(ae)$, as in Figure~\ref{fig:A}, is forbidden.
\end{lemma}

\begin{proof}
Assume towards a contradiction that for distinct vertices $a,b,c,d,e$, we have $\varphi(be)=\varphi(cd)=\alpha$, $\varphi(ab)=\varphi(de)=\beta$, and $\varphi(bc)=\varphi(ae)=\gamma$. Let $\alpha_0=(i,\{x,y\})$.
First suppose $b^{(i)}=c^{(i)}=x$ and $e^{(i)}=d^{(i)}=y$. Then $\beta_i=0$ since  $d$ and $e$ agree at $i$, and $\gamma_i=0$ since $b$ and $c$ agree at $i$. But then  $a^{(i)}=x$ (because  $\varphi(ab)=\beta$, $\beta_i=0$ and $b^{(i)}=x$), and similarly $a^{(i)}=y$ (because  $\varphi(ae)=\gamma$, $\gamma_i=0$ and $e^{(i)}=y$), a contradiction since $x\neq y$. 
So, instead suppose that $b^{(i)}=d^{(i)}=x$ and $c^{(i)}=e^{(i)}=y$. Then $\beta_i=t$, the first index of difference between $d^{(i)}=x$ and $e^{(i)}=y$. Similarly, $\gamma_i=t$. Let $a^{(i)}=w$. Then since $\varphi(ab)=\beta$ and $\beta_i=t$, we know $x_j=w_j$ for all $1\leq j\leq t-1$ and $x_t\neq w_t$. Similarly, since $\varphi(ae)=\gamma$ and $\gamma_i=t$, we have $y_j=w_j$ for all $1\leq j\leq t-1$ and $y_t\neq w_t.$ But since $\{x_t,y_t\}=\{0,1\}$, this leaves no choice for $w_t$, a contradiction. 
\end{proof}

\begin{lemma}\label{fig:1matching}
The configuration with five distinct vertices $a,b,c,d,e\in V$ for which $\varphi(ac)=\varphi(bc)$, $\varphi(ab)=\varphi(bd)$,  $\varphi(ad)=\varphi(de)$, and $\varphi(ae)=\varphi(ec)$, 
as in Figure~\ref{fig:B}, is forbidden.
\end{lemma}

\begin{proof}
Assume for contradiction that for distinct vertices $a,b,c,d,e$, we have $\varphi(ac)=\varphi(bc)=\alpha$, $\varphi(ab)=\varphi(bd)=\beta$,  $\varphi(ad)=\varphi(de)=\gamma$, and $\varphi(ae)=\varphi(ec)=\pi$.
 Let $\alpha_0=(i,\{x,y\})$, and let $c^{(i)}=x$ and $a^{(i)}=b^{(i)}=y$. Then $\beta_i=0$ since $a$ and $b$ agree at $i$, which implies $d^{(i)}=y$ as well. Now we conclude  $\gamma_i=0$ since $a$ and $d$ agree at $i$, which implies $e^{(i)}=y$ too. But now $\pi_i=0$ since $a$ and $e$ agree at $i$, while also $\pi_i\neq 0$ since $c$ and $e$ disagree at $i$, a contradiction.
\end{proof}

\begin{lemma}\label{lem:rainbow cycle plus matching}
Any 2-2-2-2-2 coloring of $K_5$  with a rainbow 5-cycle and at least one matching, as in Figure~\ref{fig:C}, is forbidden under $\varphi$.
\end{lemma}

\begin{proof}
Assume towards a contradiction that for distinct vertices $a,b,c,d,e\in V$, we have $\varphi(be)=\varphi(cd)=\alpha$, $\varphi(ab)=\beta$,  $\varphi(bc)=\gamma$, $\varphi(de)=\pi$, and $\varphi(ea)=\theta$.
Note that since this is a 2-2-2-2-2 coloring of $K_5$, we have $\{\varphi(bd),\varphi(ce)\}=\{\beta,\theta\}$, since otherwise we would find a copy of the forbidden configuration shown in Figure~\ref{fig:forb1} using the colors $\{\alpha,\gamma\}$ or $\{\alpha,\pi\}$. This implies $\{\varphi(ac),\varphi(ad)\}=\{\pi,\gamma\}$.

Suppose first that $\varphi(bd)=\beta$ and $\varphi(ce)=\theta$.  If $\varphi(ac)=\pi$ and $\varphi(ad)=\gamma$, we find the forbidden configuration shown in Figure~\ref{fig:A} in colors $\{\alpha,\gamma,\pi\}$.  If instead $\varphi(ac)=\gamma$ and $\varphi(ad)=\pi$, then we find the forbidden configuration shown in Figure~\ref{fig:B} in colors $\{\beta,\gamma,\pi,\theta\}$. 
So, we must have $\varphi(bd)=\theta$ and $\varphi(ce)=\beta$, and we get the forbidden configuration shown in Figure~\ref{fig:A} in colors $\{\alpha,\beta,\theta\}$, a contradiction.
\end{proof}

\begin{lemma}\label{lem:lastforbidden}
The configuration with five distinct vertices $a,b,c,d,e\in V$ for which $\varphi(bc)=\varphi(de)$,  $\varphi(ad)=\varphi(cd)$, and $\varphi(ab)=\varphi(ac)$ as in Figure~\ref{fig:D}, is forbidden.
\end{lemma}

\begin{proof}
Assume towards a contradiction that for distinct vertices $a,b,c,d,e$, we have $\varphi(bc)=\varphi(de)=\alpha$, $\varphi(ad)=\varphi(cd)=\beta$, and $\varphi(ab)=\varphi(ac)=\gamma$. Let $\alpha_0=(i,\{x,y\})$. First suppose $b^{(i)}=e^{(i)}=x$ and $c^{(i)}=d^{(i)}=y$.  Then $\beta_i=0$ since $b$ and $c$ agree at $i$, and hence $a^{(i)}=y$ as well. But now $\gamma_i=0$ since $a$ and $c$ agree at $i$ while also $\gamma_i\neq 0$ since $a$ and $b$ disagree at $i$, a contradiction. So, it must be the case that $b^{(i)}=d^{(i)}=x$ and $c^{(i)}=e^{(i)}=y$.

Now $\beta_i=t$, the first index of difference between $c^{(i)}=y$ and $d^{(i)}=x$. Say $a^{(i)}=w$. Since $\varphi(ad)=\beta$, we know that $x_j=w_j$ for all $1\leq j\leq t-1$ and $x_t\neq w_t$. Then since $x_t\neq y_t$ as well, we have $w_t=y_t$. Furthermore, 
$\beta_i=t$ implies that $w_j=x_j=y_j$ for all $1\leq j\leq t-1$, so the first index of difference between $w$ and $y$ must be greater than $t$. Now consider $\gamma_i$. Since $\varphi(ac)=\gamma$, we know $\gamma_i>t$. But $\varphi(ab)=\gamma$ implies that $\gamma_i=t$, a contradiction. 
\end{proof}

\begin{lemma}\label{lem:rotated narrow paths}
The configuration with five distinct vertices $a,b,c,d,e\in V$ for which $\varphi(ab)=\varphi(be)$,  $\varphi(bc)=\varphi(ac)$, $\varphi(cd)=\varphi(bd)$, $\varphi(de)=\varphi(ce)$, and $\varphi(ae)=\varphi(ad)$, as in Figure~\ref{fig:E}, is forbidden.
\end{lemma}

\begin{proof}
Assume towards a contradiction that for distinct vertices $a,b,c,d,e$, we have $\varphi(ab)=\varphi(be)=\alpha$,  $\varphi(bc)=\varphi(ac)=\beta$, $\varphi(cd)=\varphi(bd)=\gamma$, $\varphi(de)=\varphi(ce)=\pi$, and $\varphi(ae)=\varphi(ad)=\theta$. 
Let $\alpha_0=(i,\{x,y\})$, and let $a^{(i)}=e^{(i)}=x$ and $b^{(i)}=y$. Then $\theta_i=0$ since $a$ and $e$ agree at $i$, which implies $d^{(i)}=x$ as well. Similarly, $\pi_i=0$ since $d$ and $e$ agree at $i$, which forces $c^{(i)}=x$. But now $\gamma_i=0$ because $c$ and $d$ agree at $i$, while $\gamma_i\neq 0$ since $b$ and $d$ disagree at $i$, a contradiction.
\end{proof}

Finally, we are ready to  show that every 2-2-2-2-2 coloring of $K_5$ is forbidden by $\varphi$. We will  do this by splitting our proof into cases based on the number of matchings and paths in a 2-2-2-2-2 configuration.

\begin{lemma}\label{lem:5paths} There is no 2-2-2-2-2 configuration in $\varphi$ in which all five color classes are paths of length 2.
\end{lemma}
\begin{proof}
Assume for contradiction that such a coloring of $K_5$ exists under $\varphi$. First, we will show that this $K_5$ must contain a rainbow cycle. To this end, note that each monochromatic path of length 2 is contained in two 5-cycles in our $K_5$. So if all the 5 colors are length 2 paths, there are at most 10 cycles in $K_5$ that are not rainbow. But since there are 12 distinct 5-cycles in $K_5$, there must be at least two 5-cycles containing no monochromatic path; that is, there must be some rainbow cycle.

Consider this rainbow cycle and suppose  $\varphi(ab)=\alpha$,  $\varphi(bc)=\beta$, 
$\varphi(cd)=\gamma$, $\varphi(de)=\pi$, and $\varphi(ea)=\theta$. Since the color class of $\alpha$ forms a length 2 path, 
 there are two cases up to symmetry: either 
$\varphi(be)=\alpha$, or $\varphi(bd)=\alpha$. 

First suppose $\varphi(be)=\alpha$. 
Note that $\varphi(bd)\in\{\gamma,\pi\}$, because by Lemma \ref{neighborhood3}, $b$  must be adjacent to edges of at least three colors, and we know each color class is a path. This forces $\varphi(ad)=\theta$ because  $d$ must be adjacent to edges of at least three colors by Lemma~\ref{neighborhood3}, and since $\beta$  forms a path. 
If $\varphi(bd)=\pi$, then $\{\varphi(ac), \varphi(ce)\}=\{\beta,\gamma\}$, contradicting Lemma~\ref{neighborhood3} at $c$.  So, it must be the case that $\varphi(bd)=\gamma$ and $\{\varphi(ac), \varphi(ce)\}=\{\beta,\pi\}$. Since $\pi$ forms a path, this implies $\varphi(ac)=\beta$ and $\varphi(ce)=\pi$. However, this is  the forbidden configuration in Figure~\ref{fig:E}, so we reach a contradiction.

Now instead suppose $\varphi(bd)=\alpha$. By Lemma~\ref{neighborhood3}, and because $\gamma$ must form a path, we have $\varphi(be)\in\{\pi,\theta\}$. Then $\varphi(ce)\in\{\beta,\gamma\}$ since Lemma~\ref{neighborhood3} requires at least three colors on the edges incident to $e$. This forces $\varphi(ac)=\theta$, which in turns implies $\varphi(be)=\pi$. Finally, since $\beta$ must be a path, we have $\varphi(ad)=\gamma$ and $\varphi(ce)=\beta$. But now we again have the forbidden configuration in Figure~\ref{fig:E}, a contradiction.
\end{proof}

\begin{lemma}\label{lem:1matching4paths} There is no 2-2-2-2-2 configuration in $\varphi$ in which one color classes is a matching and the rest are paths.
\end{lemma}

\begin{proof}
Assume towards a contradiction that distinct vertices $a,b,c,d,e\in V$ form a 2-2-2-2-2 coloring of $K_5$ where one color class is a matching and four are paths. Let $\varphi(ab)=\varphi(cd)=\alpha$. Consider the path in color $\beta$. If $e$ is incident to two edges of color $\beta$, then by Lemma \ref{neighborhood3}, the other two edges incident to $e$ must receive distinct colors, say $\gamma$ and $\pi$. Now the fifth color $\theta$ must appear on a path in the subgraph induced by $\{a,b,c,d\}$. But this forms a forbidden configuration as in Figure~\ref{fig:forb1} in colors $\theta$ and $\alpha$, a contradiction. Since the same argument applies to $\gamma$, $\pi$, and $\theta$, we know $e$ is incident to edges of four distinct colors.
However, there is only one such configuration in which the color classes of $\beta,\gamma, \pi$, and $\theta$ are all paths, and it is the forbidden configuration shown in Figure~\ref{fig:B}, so we again reach a contradiction.
\end{proof}

\begin{lemma}\label{lem:2matchings3paths} There is no 2-2-2-2-2 configuration in $\varphi$ in which two color classes are matchings and three are paths.
\end{lemma}
\begin{proof}
Assume for contradiction that distinct vertices $a,b,c,d,e\in V$ form a 2-2-2-2-2 coloring of $K_5$ with  matchings in colors $\alpha$ and $\beta$ and  paths in colors $\gamma, \pi, \theta$. If $\alpha$ and $\beta$ form an alternating path, say $\varphi(ab)=\varphi(cd)=\alpha$ and $\varphi(bc)=\varphi(de)=\beta$, then consider the edge $bd$. Without loss of generality, $\varphi(bd)=\gamma$. If $\varphi(be)=\gamma$ or $\varphi(ad)=\gamma$, then there is a copy of the forbidden configuration in Figure~\ref{fig:forb1}. So, the other edge of color $\gamma$ must be one of $ac, ce,$ or $ae$. But in these cases, $\gamma$ would form a matching rather than a path, a contradiction. 

Now suppose $\alpha$ and $\beta$ form an alternating cycle, say $\varphi(ab)=\varphi(cd)=\alpha$ and $\varphi(bc)=\varphi(ad)=\beta$. Since all remaining edges receives colors in $\{\gamma, \pi,\theta\}$, we know that the  two edges incident to $e$ must receive the same color. Without loss of generality, say $\varphi(ae)=\varphi(be)=\gamma$. Lemma \ref{neighborhood3} implies that $\pi$ and $\theta$ each appear on one of the edges $ce$ and $de$.  Furthermore, since the color classes of $\pi$ and $\theta$ are paths, we may assume that $\varphi(ac)=\varphi(ce)=\pi$ and $\varphi(bd)=\varphi(de)=\theta$. But now the colors $\alpha, \gamma$, and $\pi$ form a forbidden configuration as in Figure~\ref{fig:D}, a contradiction. 
\end{proof}

\begin{lemma}\label{lem:3matchings2paths} There is no 2-2-2-2-2 configuration in $\varphi$ in which three color classes are matchings and two are paths.
\end{lemma}

\begin{proof}
Assume for contradiction that distinct vertices $a,b,c,d,e\in V$ form a 2-2-2-2-2 coloring of $K_5$ with three color classes that are matchings and two color classes that are paths. Let $\varphi(ab)=\varphi(bc)=\alpha$ and $\varphi(ac)=\beta$. Then either $\beta$ is a path, or a matching.

First suppose $\beta$ forms a path. Without loss of generality, $\varphi(cd)=\beta$. Then $ad$ receives some third color, say $\varphi(ad)=\gamma$, and $\gamma$ forms a matching. We cannot have $\varphi(ec)=\gamma$ without forming a forbidden configuration as in Figure~\ref{fig:forb1} with colors $\{\beta, \gamma\}$, so $\varphi(be)=\gamma$ is forced. Now the remaining three edges incident to $e$  should be colored with the two colors $\pi$ and $\theta$, contradicting the fact that $\pi$ and $\theta$ both form a matching each.

Therefore, $\beta$ must form a matching. Since we must avoid the forbidden configuration in Figure~\ref{fig:forb1}, the other edge colored $\beta$ cannot be incident to $b$, which implies  $\varphi(de)=\beta$. Consider where the second monochromatic path, say in color $\gamma$, can occur. If $b$ is not incident to an edge of color $\gamma$, then we find a forbidden configuration as in Figure~\ref{fig:forb1} in colors $\{\beta,\gamma\}$. So, $b$ must be incident to an edge of color $\gamma$, say $\varphi(bd)=\gamma$. By Lemma~\ref{neighborhood3}, $\varphi(be)\neq \gamma$, so either $cd$ or $ad$ also receive color $\gamma$. 

In both cases, there are three remaining edges incident to $e$ which must be colored with $\{\pi,\theta\}$, but  both $\pi$ and $\theta$ form matchings in the $K_5$, a contradiction. 
\end{proof}

\begin{lemma}\label{lem:4matchings1path} There is no 2-2-2-2-2 configuration in $\varphi$ in which four color classes are matchings and one is a path.
\end{lemma}

\begin{proof}
 Assume towards a contradiction that distinct vertices $a,b,c,d,e\in V$ form a 2-2-2-2-2 coloring of $K_5$ with a path in color $\alpha$ and four matchings.  Let $\varphi(ab)=\varphi(bc)=\alpha$, and say $\varphi(bc)=\beta$. Then $\beta$ cannot appear on the edges $bd$ or $be$, otherwise there is a copy of the forbidden configuration in Figure~\ref{fig:forb1}. So, since $\beta$ forms a matching, we must have $\varphi(de)=\beta$. Now the remaining six edges in this $K_5$ form a copy of $K_{2,3}$ on vertex set $\{a,b,c\}\cup \{d,e\}$ which must be colored with three colors, but this cannot occur when all three color classes are matchings of size two. 
\end{proof}

\begin{lemma}\label{lem:5matchings} There is no 2-2-2-2-2 configuration in $\varphi$ in which all five color classes are matchings.
\end{lemma}

\begin{proof}
Assume towards a contradiction that distinct vertices $a,b,c,d,e\in V$ form a 2-2-2-2-2 coloring of $K_5$ where each color class is a matching. Without loss of generality, assume $\varphi(ab)=\varphi(cd)=\alpha$ and $\varphi(bc)=\beta$. Then up to symmetry, there are two cases: either $\varphi(ad)=\beta$ or $\varphi(de)=\beta$. 

Consider the case $\varphi(ad)=\beta$. Let $\varphi(ae)=\gamma$. Then since $\gamma$ is a matching, we must have $\varphi(bd)=\gamma$ as well. But then the remaining two colors cannot both form matchings, since two of the edges $be, ce,$ and $de$ must receive the same color by pigeon-hole principle. 

Thus, we may assume $\varphi(de)=\beta$ instead. Let $\varphi(ae)=\gamma$. Then since $\gamma$ must form a matching, we have $\varphi(bd)=\gamma$ as well. However, this forms the forbidden configuration shown in Figure~\ref{fig:A}, a contradiction. 
\end{proof}

Combining the results of Lemmas~\ref{lem:5paths}-\ref{lem:5matchings} shows that $\varphi$ forbids all 2-2-2-2-2 colorings of $K_5$. Therefore, any $K_5$ which appears under $\varphi$ must contain some color an odd number of times. 
\qed

\section{Forbidden submatching method}\label{sec:blackbox}

We will use the simplified version presented in~\cite{JM} of the conflict-free hypergraph matching theorem from~\cite{GJKKL} to prove Theorem~\ref{thm:C4inKnn}. In order to state this theorem, we will need to introduce some terminology and notation. 

Given a hypergraph $\HH$ and a vertex $v\in V(\HH)$, its \emph{degree} $\deg_{\HH}(v)$ 
is the number of edges in $\HH$ containing $v$. The maximum degree and minimum degree of $\HH$ are denoted by $\Delta(\HH)$ and $\delta(\HH)$, respectively. For $j\geq 2$, $\Delta_j(\HH)$ denotes the maximum number of edges in $\HH$ which contain a particular set of $j$ vertices, over all such sets. 

In addition, for a (not necessarily uniform) hypergraph $\CC$ and an integer $k$, let $\CC^{(k)}$ be the set of edges in $\CC$ of size $k$. For a vertex $u\in V(\CC)$, let $\CC_u$ denote the hypergraph $\{C\backslash \{u\} \mid C\in E(\CC), u\in C\}$.

Given a hypergraph $\HH$, a hypergraph $\CC$ is a \emph{conflict system} for $\HH$ if $V(\CC)=E(\HH)$. A set of edges $E \subset \HH$ is \emph{$\CC$-free} if $E$ contains no subset $C\in \CC$. Given integers $d\geq 1$, $\ell\geq 3$, and $\eps\in(0,1)$, we say $\CC$ is \emph{$(d,\ell,\eps)$-bounded} if  $\CC$ satisfies the following conditions:
\begin{enumerate}
    \item[(C1)] $3\leq |C|\leq \ell$ for all $C\in\CC$;
    \item[(C2)] $\Delta(\CC^{(j)})\leq \ell d^{j-1}$ for all $3\leq j\leq \ell$;
    \item[(C3)] $\Delta_{j'}(\CC^{(j)})\leq d^{j-j'-\eps}$ for all $2\leq j'< j\leq \ell$. 
\end{enumerate}

Finally, given a $(d,\ell,\eps)$-bounded conflict system $\CC$ for a hypergraph $\HH$, we will define a type of weight function which can be used to guarantee that the almost-perfect matching given by Theorem~\ref{thm:blackbox} below satisfies certain quasirandom properties. We say a function $w:\binom{\HH}{j}\rightarrow[0,\ell]$ for $j\in\naturals$ is a \emph{test function} for $\HH$ if $w(E)=0$ whenever $E\in\binom{\HH}{j}$ is not a matching, and we say $w$ is \emph{$j$-uniform}. For a function $w:A\rightarrow \reals$ and a finite set $X\subset A$, let $w(X):=\sum_{x\in X} w(x)$. If $w$ is a $j$-uniform test function, then for each $E\subset \HH$, let $w(E)=w(\binom{E}{j})$. Given $j,d\in\naturals$, $\eps>0$, and a conflict system $\CC$ for hypergraph $\HH$, we say a $j$-uniform test function $w$ for $\HH$ is \emph{$(d,\eps,\CC)$-trackable} if $w$ satisfies the following conditions:
\begin{enumerate}
    \item[(W1)] $w(\HH)\geq d^{j+\eps}$;
    \item[(W2)] $w(\{{E\in \binom{\HH}{j}:E\supseteq E'\})\leq w(\HH})/d^{j'+\eps}$ for all $j'\in[j-1]$ and $E'\in\binom{\HH}{j'}$;
    \item[(W3)] $|(\CC_e)^{(j')}\cap(\CC_f)^{(j')}|\leq d^{j'-\eps}$ for all $e,f\in \HH$ with $w(\{E\in\binom{\HH}{j}:e,f\in E\})>0$ and all $j'\in[\ell-1]$;
    \item[(W4)] $w(E)=0$ for all $E\in\binom{\HH}{j}$ that are not $\CC$-free.
\end{enumerate}

\begin{theorem}[\cite{GJKKL}, Theorem 3.3]\label{thm:blackbox}
For all $k,\ell\geq 2$, there exists $\eps_0>0$ such that for all $\eps\in(0,\eps_0)$, there exists $d_0$ such that the following holds for all $d\geq d_0$. Suppose $\HH$ is a $k$-regular hypergraph on $n\leq \exp(d^{\eps^3})$ vertices with $(1-d^{-\eps})d\leq \delta(\HH)\leq \Delta(\HH)\leq d$ and $\Delta_2(\HH)\leq d^{1-\eps}$. Suppose $\CC$ is a $(d,\ell,\eps)$-bounded conflict system for $\HH$, and suppose $\WW$ is a set of $(d,\eps, \CC)$-trackable test functions for $\HH$ of uniformity at most $\ell$ with $|\WW|\leq \exp(d^{\eps^3})$. Then, there exists a $\CC$-free matching $\MM\subset \HH$ of size at least $(1-d^{-\eps^3})n/k$ with $w(\MM)=(1\pm d^{-\eps^3})d^{-j}w(\HH))$ for all $j$-uniform $w\in \WW$. 
\end{theorem}

We will say that a hypergraph $\HH$ with $(1-d^{-\eps})d\leq \delta(\HH)\leq \Delta(\HH)\leq d$  is \emph{almost $d$-regular.}

\section{Proof of Theorem~\ref{thm:C4inKnn}}\label{sec:C4inKnn}

First, observe that $g(K_{n,n},C_4)>n/2$. Indeed, let $V(K_{n,n})=X\cup Y$ and suppose the edges of $K_{n,n}$ are colored with at most $n/2$ colors. Then for each vertex $x\in X$, there are at least $n/2$ pairs $u,v\in Y$ such that $xu$ and $xv$ receive the same color. Since there are $\binom{n}{2}$ pairs of vertices in $Y$, the pigeonhole principle implies that there must be two vertices $x,z\in X$ such that $xu$ and $xv$ share a color and $zu$ and $zv$ share a color, giving a copy of $C_4$ with no color appearing an odd number of times. 

From now on, we remove the ceiling function notation in order to make the proof easier to read. However, whenever $n$ is odd, when we write $\frac{n}{2}$ we mean $\lceil \frac{n}{2} \rceil$. 

In order to prove the upper bound in Theorem~\ref{thm:C4inKnn}, 
we will construct a coloring of $K_{n,n}$ in two stages. The coloring in the first stage will use $k=n/2$ colors to color a majority of the edges of $K_{n,n}$. To determine this coloring, we define appropriate hypergraphs $\HH$ and $\CC$ for which a $\CC$-free matching in $\HH$ corresponds to a partial coloring of $K_{n,n}$ with monochromatic copies of $C_6$. Each copy of $C_4$ in the resulting coloring sees some color an odd number of times. 

Theorem~\ref{thm:coloringproperties} below guarantees that our choices of $\HH$ and $\CC$ satisfy the requirements for applying the forbidding submatching method to find the desired matching. In the second stage of the coloring, we then randomly color the remaining uncolored edges using a new set of $o(n)$ colors. 

\begin{theorem}\label{thm:coloringproperties}
There exists $\delta>0$ such that for all sufficiently large $n$ in terms of $\delta$, there is an edge-coloring of a subgraph $F\subset K_{n,n}$ with at most $n/2$ colors and the following properties:
\begin{enumerate}
    \item Every color class is a union of vertex-disjoint copies of $C_6$.
    \item Every copy of $C_4$ in $F$ intersects some color class in an odd number of edges.
    \item The graph $L=K_{n,n}\setminus E(F)$ has maximum degree at most $n^{1-\delta}$.
    \item Let $V(K_{n,n})=X\cup Y$. For each $(x,y) \in X \times Y$, the number of edges $x'y'\in E(L)$ with $x'\in X \backslash\{x\}$, $y'\in Y \backslash\{y\}$ such that $xy'$ and $yx'$ receive the same color is at most $4n^{1-\delta}$.
\end{enumerate}
\end{theorem}

\begin{proof}
 Let $U=\binom{X\cup Y}{2}$ and $V=\bigcup_{i\in [\frac{n}{2}]}V_i$, where each $V_i$ is a copy of $X\cup Y$. We will denote the copy of $v\in X\cup Y$ in $V_i$ by $v^i$.  Let $\HH$ be the 18-uniform hypergraph with vertex set $U\cup V$ and edges $e_{i}$ for each copy $e$ of $C_6$ in $K_{n,n}$ and $i\in [\frac{n}{2}]$ defined as follows. Given a 6-cycle $e=x_1y_1x_2y_2x_3y_3$ with $x_j\in X$ and $y_j\in Y$, let $\HH$ contain the edge
\[e_{i}=\{x_1y_1,y_1x_2,x_2y_2,y_2x_3,x_3y_3,y_3x_1\}\cup\{x_1x_2,x_2x_3,x_1x_3,y_1y_2,y_2y_3,y_1y_3\}\cup\{x_1^i,x_2^i,x_3^i,y_1^i,y_2^i,y_3^i\}.\] Note that a matching in $\HH$ corresponds to a set of edge-disjoint monochromatic copies of $C_6$ in $K_{n,n}$.  Furthermore, no vertex in the matching appears in two cycles of the same color. Thus, a matching in $\HH$ yields an edge-coloring of a subgraph of $K_{n,n}$ with at most $\frac{n}{2}$ colors which satisfies property (1) above. 

The hypergraph $\HH$ is essentially $d$-regular for $d=\frac{1}{2}n^5$. Indeed, we can check that for each vertex $u$ in $\HH$,  $d-O(n^4)\leq \deg_{\HH}(u)\leq d$. First, suppose $u\in V_i$ for some $i$. There are  $n^3\binom{n}{2}-O(n^4)$ ways to pick an edge in $\HH$ containing $u$, since there are $\binom{n}{2}$ choices for the neighbors of $u$ and $n-O(1)$ choices for each of the remaining vertices along the cycle.  If instead $u=xy$ for some $x\in X$ and $y\in Y$, then there are $\frac 12 n$ ways to pick a color of an edge containing $xy$ and $n^4-O(n^3)$ ways to pick the remaining vertices along the cycle. Similarly, if $u=xx'$ for $x,x'\in X$ (or $yy'$ for $y,y'\in Y$), there are $\frac 12 n^5-O(n^4)$ edges in $\HH$ containing the edge. Thus, in each case, we have $\deg_{\HH}(u)=\frac{1}{2}n^5-O(n^4)$. 

Furthermore, we have $\Delta_2(\HH)\leq d^{1-\eps}$ for all $\eps \in (0,\frac{1}{6})$, since each pair of vertices $u,v\in V(\HH)$ is contained in at most $O(n^4)$ edges in $\HH$.

Next we define a conflict system $\CC$ for $\HH$ with edges of size 3 and 4.  Let $V(\CC)=E(\HH)$, and let the edges of $\CC$ correspond to copies of $C_4$ in $K_{n,n}$ formed by two monochromatic matchings of size 2.  We will call these edges of $\CC$ \emph{conflicts}, and they correspond to sets of three or four monochromatic cycles in $K_{n,n}$.  That is, given four vertices $x,x'\in X$ and $y,y'\in Y$, colors $i,j\in [\frac{n}{2}]$ and edges $e_{i}, e_{i}', f_{j},f_{j}'\in E(\HH)$ (not necessarily distinct) with $xy\in e_{i}$, $x'y'\in e'_{i}$, $xy'\in f_j$, and $x'y\in f'_j$, we have $\{e_{i}, e'_{i}, f_j,f'_j\}\in E(\CC)$. Note that we do not need to consider conflicts of size 2:   it is possible to have two monochromatic copies $e,e'$ of $C_6$ that  form a copy of $C_4$ with alternating colors, however such $e,e'$ would have to share two vertices in $X$ (and two vertices in $Y$), and thus cannot appear together in a matching of $\HH$.

We will show that $\Delta(\CC^{(3)})=O(d^2)$ and $\Delta(\CC^{(4)})=O(d^3)$. To this end, fix $e_{i}\in V(\CC)$. Note that any conflict in $\CC^{(3)}$ containing $e_i$ must either have the form $\{e_{i},e_{i}',f_j\}$ or $\{e_{i},f_j, f'_j\}$ for some copies $e',f,f'$ of $C_6$ in $K_{n,n}$ and some color $j\in[\frac{n}{2}]\backslash\{i\}$. In either case, there are $O(n)$ ways to pick $j$. If the edge has the form $\{e_i,e_i',f_j\}$, then there are $O(n^6)$ ways to pick $e'$ and then $O(n^2)$ ways to pick $f$. If instead the edge has the form $\{e_i,f_j,f_j'\}$, then there are $O(n^4)$ ways to pick $f$ and $O(n^4)$ ways to pick $f'$. Thus, in either case, $d_{\CC}(e_i)=O(n^9)$, and we have $\Delta(\CC^{(3)})=O(d^2)$. Similarly, any conflict in $\CC^{(4)}$ containing $e_i$ must have the form $\{e_i,e'_i,f_j,f'_j\}$ for some copies $e, e',f,f'$ of $C_6$ and some color $j\in[\frac{n}{2}]$, so there are $O(n)$ choices for $j$, $O(n^6)$ choices for $e'$, and $O(n^4)$ choices for each of $f$ and $f'$. Thus, $\Delta(\CC^{(4)})=O(n^{15})=O(d^3)$. 

In addition, we have $\Delta_2(\CC^{(3)})<d^{1-\eps}$ for all $\eps\in(0,\frac{1}{6})$. Indeed, the number of conflicts in $\CC^{(3)}$ containing two fixed vertices $e_i,e_i'\in V(\CC)$ is $O(n^3)$, since there are $O(n)$ ways to pick another color and $O(n^2)$ ways to pick two more vertices for the third edge in the conflict. Similarly, the number of conflicts in $\CC^{(3)}$ containing two fixed vertices $e_i,f_j\in V(\CC)$ is $O(n^4)$, which is the
number of ways to pick four more vertices for the third edge in the conflict. Similarly, we have $\Delta_2(\CC^{(4)})<d^{2-\eps}$ for all $\eps\in(0,\frac{1}{6})$. The number of conflicts in $\CC^{(4)}$ containing two fixed vertices $e_i,e_i'\in V(\CC)$ is $O(n^{9})$ since there are $O(n)$ choices for a second color and $O(n^4)$ choices for each of $f$ and $f'$ to complete the conflict, and the number of conflicts in $\CC^{(4)}$ containing two fixed vertices $e_i,f_j\in V(\CC)$ is $O(n^{9})$ since there are $O(n^5)$ choices for $e'$ and $O(n^4)$ choices for $f'$ to complete the conflict. Finally, $\Delta_3(\CC^{(4)})<d^{1-\eps}$, because given three vertices $e_i,e'_i,f_j\in V(\CC)$, there are $O(n^4)$ choices for $f'$ to complete a conflict. Therefore, $\CC$ is a $(d,O(1),\eps)$-bounded conflict system for $\HH$ for all $\varepsilon\in(0,\frac{1}{6})$. 

Note that we could apply Theorem~\ref{thm:blackbox} at this point to obtain a conflict-free matching $M$ in $\HH$ which would correspond to a coloring of a subgraph $F$ of $K_{n,n}$ satisfying properties (1) and (2) of Theorem~\ref{thm:coloringproperties}. 
In order to obtain a coloring which also satisfies properties (3) and (4)
of Theorem~\ref{thm:coloringproperties}, we now introduce appropriate test functions. 

First, we define a set of 1-uniform test functions as in~\cite{JM}. For each $v\in X\cup Y$, let $S_v\subset U$ be the set of edges incident to $v$ in $K_{n,n}$. Let $w_v:E(\HH)\rightarrow\{0,2\}$ be the weight function which assigns to each edge $e_i\in E(\HH)$ the size of its intersection with $S_v$. 
Then we have  $$w_v(\HH)=\sum_{e_i\in \HH}w_v(e_i)
=
\sum_{f\in S_v}\deg_{\HH}(f)=nd-O(n^{5})\geq d^{1+\eps},$$ 
satisfying condition (W1). In addition, conditions (W2)-(W4) are trivially satisfied since $w_v$ is 1-uniform, so $w_v$ is a $(d,\eps,\CC)$-trackable  test function for all $\eps\in(0,\frac{1}{6})$ and $v\in X\cup Y$. 

Let $\delta<\eps^3\log_n(\frac{1}{2}n^5)$. Note that applying Theorem~\ref{thm:blackbox} with these test functions would yield a $\CC$-free matching $M\subset \HH$ such that for each $v\in X\cup Y$, 
\[w_v(M)>(1-d^{-\eps^3})d^{-1}w_v(\HH)>(1-n^{-\delta})n.\] 
Thus, for each $v\in X\cup Y$, there are at most $n^{1-\delta}$ edges in $K_{n,n}$ incident to $v$ which do not belong to any edge in $M$. Hence, property (3) is satisfied, as the graph $L$ containing the uncolored edges has maximum degree $\Delta(L)\leq n^{1-\delta}$. 

In order to guarantee that property (4) is also satisfied, we will define four more types of test functions.  The first two will help us to show that for each $(x,y) \in X \times Y$, there are at most $n^{1-\delta}$ edges $x'y'\in E(L)$ with $x'\in X \backslash\{x\}$ and $y'\in Y \backslash\{y\}$ such that $xy'$ and $yx'$ are in distinct edges of the same color in the matching $M\subset \HH$. The next two will help us to show that for each $(x,y)\in X\times Y$, there are at most $n^{1-\delta}$ edges in the matching that contain both $x$ and $y$, and hence at most $3n^{1-\delta}$ graph edges 
$x'y'\in E(L)$ with $x'\in X\backslash \{x\}$ and $y'\in Y\backslash \{y\}$ such that $xy$ and $x'y'$ are contained in the same edge of the matching.

For each pair of vertices $(x,y)\in X\times Y$, let
\[\PP_{x,y}=\left\{\{e_i,e'_i\}:i\in\left[\frac{n}{2}\right], V(e)\cap V(e')=\emptyset,
x\in V(e), y\in V(e')\right\}.\]
Note that
\[|\PP_{x,y}|=\frac{n}{2}\cdot \frac{1}{2}n^5\cdot \frac{1}{2}n^5 \pm O(n^{10})=\frac{1}{8}n^{11}\pm O(n^{10})>d^{2+\eps}.\]

For each $(x,y)\in X\times Y$, define $w_{x,y}$ to be the indicator weight function for the pairs in $\PP_{x,y}$. Since each pair in $\PP_{x,y}$ is a matching of size 2 in $\HH$, $w_{x,y}$ is a 2-uniform test function.
We will show that $w_{x,y}$ is $(d,\eps,\CC)$-trackable for all $\eps\in(0,\frac{1}{6})$, but first we will define a set of 3-uniform test functions.

For each pair of vertices $(x,y)\in X\times Y$, let
\[\TT_{x,y}=\left\{\{e_i,e'_i,f_j\}:\{e_i,e'_i\}\in\PP_{x,y}, j\in\left[\frac{n}{2}\right]\backslash\{i\}, |V(f)\cap V(e)\cap Y|=1, |V(f)\cap V(e')\cap X|=1\right\}.\] 
Thus, the triples in $\TT_{x,y}$ are matchings of size 3 in $\HH$. Note that 
\[|\TT_{x,y}|=9(d\pm O(n^4))|\PP_{x,y}|.\]

Similarly to above, define for each $(x,y)\in X\times Y$ the function $w'_{x,y}$ to be the indicator weight function for triples in $\TT_{x,y}$. This is a 3-uniform test function for $\HH$.

Assuming that these test functions $w_{x,y}$ and $w'_{x,y}$ are $(d,\eps,\CC)$-trackable, we can apply Theorem~\ref{thm:blackbox} to obtain a matching $M$ in $\HH$ such that 
\[w_{x,y}(M)=\left|\binom{M}{2}\cap \PP_{x,y}\right|\leq (1+ d^{-\eps^3})d^{-2}|\PP_{x,y}|\leq (1+n^{-2\delta})\frac{n}{2}\]
and 
\[w'_{x,y}(M)=\left|\binom{M}{3}\cap \TT_{x,y}\right|\geq (1-d^{-\eps^3})d^{-3}|\TT_{x,y}|\geq (1-n^{-2\delta})\frac{9n}{2}.\]
Note that each edge $x'y'\in E(L)$ with $xy'$ and $x'y$ appearing in distinct edges of $\HH$ of the same color corresponds to an edge of $\HH$ in a triple of $\TT_{x,y}$ extending some pair from $\PP_{x,y}$. Thus, we can bound the number of such $x'y'$ as follows:
\[9\left|\binom{M}{2}\cap \PP_{x,y}\right|-\left|\binom{M}{3}\cap\TT_{x,y}\right|\leq n^{1-\delta}.\]
Therefore, in order to prove the first case of property (4), we must verify that $w_{x,y}$ and $w'_{x,y}$ are $(d,\eps,\CC)$-trackable for $\eps\in(0,\frac{1}{6})$. For this we must check conditions (W1)-(W4) for both functions. 

Condition (W1) for $w_{x,y}$ is satisfied since $w_{x,y}(\HH)=|\PP_{x,y}|>d^{2+\eps}$. To check condition (W2), fix an edge $e_i\in \HH$. The number of  edges $e'_i$ which form a pair in $\PP_{x,y}$ with $e_i$ is at most $O(n^5)<n^{11}/d^{1+\eps}$. For condition (W3), note that the only pairs of edges $\{e,f\}\in\binom{\HH}{2}$ for which $w_{x,y}(\{e,f\})>0$ are pairs $\{e_i,e'_i\}$ with $x\in e_i$ and $y\in e'_i$. Given such a pair, the number of triples of edges in $\HH$ which can form a conflict with $e_i$ or $e'_i$ is at most $O(n\cdot n^6\cdot n^3\cdot n^3)=O(n^{13})<d^{3-\eps}.$ Finally, condition (W4) is vacuously true, so $w_{x,y}$ is $(d,\eps,\CC)$-trackable.

We now verify that $w'_{x,y}$ is $(d,\eps,\CC)$-trackable for $\eps\in(0,\frac{1}{6})$ as well. Condition (W1) holds because $w'_{x,y}(\HH)=|\TT_{x,y}|=\Theta(n^{16})>d^{3+\eps}$. To see that condition (W2) holds, fix $e_i\in\HH$. There are $O(n^{11})<w'_{x,y}(\HH)/d^{1+\eps}$ triples in $\TT_{x,y}$ containing $e_i$ since we must either pick a second color $j$, six more vertices for $e_i'$, and four more vertices for $f_j$ to form a triple $\{e_i,e_i',f_j\}$ or a second color $j$, five more vertices for $f_j$, and five more vertices for $f_j'$ to form a triple $\{e_i,f_j,f_j'\}$. And, for each pair of edges $\{e_i,e_i'\}$ or $\{e_i,f_j\}$ in $\HH$, there are $O(n^5)<w'_{x,y}(\HH)/d^{2+\eps}$ triples containing this pair. 
For condition (W3), fix a pair of edges $\{e,f\}\in \binom{\HH}{2}$ which are in at least one triple together in $\TT_{x,y}$. Either we have $\{e,f\}=\{e_i,e'_i\}$ or $\{e,f\}=\{e_i,f_j\}$. In the first case, as above, there are $O(n^{13})$ triples of edges in $\HH$ which form a conflict with either $e_i$ or $e'_i$. In the second case, there are at most $O(n^{14})<d^{3-\eps}$  conflicts containing $e_i$ which use only the colors $i$ and $j$, and this bounds the number of triples which form a conflict with either $e_i$ or $f_j$. 
Thus, condition (W3) holds. Finally, condition (W4) is vacuously true.

We now define for each $(x, y) \in X \times Y$ another pair of test functions to guarantee that there are at most $n^{1-\delta}$ edges in the matching $M\in \HH$ which contain both $x$ and $y$. Set \[\SSS_x=\left\{e_i:i\in\left[\frac{n}{2}\right],x\in V(e)\right\}\] and \[\DD_{x,y}=\left\{e_i:i\in\left[\frac{n}{2}\right],x\in V(e),y\notin V(e)\right\}.\] Let $u_x$ and $u'_{x,y}$ be the indicator functions for $\SSS_x$ and $\DD_{x,y}$, respectively. Note that $u_x(\HH)=|\SSS_x|=\frac{1}{2}n\cdot \frac{1}{2}n^5\pm O(n^5)$, and similarly $u'_{x,y}=|\DD_{x,y}|=\frac{1}{4}n^6\pm O(n^5)$. Thus, property (W1) holds for both $u_x$ and $u'_{x,y}$. Furthermore, (W2)-(W4) are vacuously true since these are 1-uniform test functions, hence both are $(d,\eps,\CC)$-trackable for $\eps\in (0,\frac{1}{6})$.

Applying Theorem~\ref{thm:blackbox} with all of these test functions gives a $\CC$-free matching $M\subset \HH$ such that for each $x\in X$, \[u_x(M)\leq (1+d^{-\eps^3})d^{-1}u_x(\HH)\leq (1+n^{-2\delta})\frac{n}{2},\] and for each $y\in Y$, 
\[u'_{x,y}(M)\geq(1-d^{-\eps^3})d^{-1}|\DD_{x,y}|\geq (1-n^{-2\delta})\frac{n}{2}.\]
Therefore, the number of edges in $M$ which contain both $x$ and $y$ is at most $u_x(M)-u'_{x,y}(M)\leq n^{1-\delta}$, as desired. Thus, the coloring obtained from $M$ satisfies property (4). 
\end{proof}

We now proceed with the proof of Theorem~\ref{thm:C4inKnn}. Applying Theorem~\ref{thm:coloringproperties} with $2\delta$ in place of $\delta$, we obtain a partial coloring of $K_{n,n}$ using at most $n/2$ colors for which the uncolored subgraph $L$ of $K_{n,n}$ has maximum degree at most $n^{1-2\delta}$. To color these remaining edges, we randomly color with a new set $P$ of $k=n^{1-\delta}$ colors, with each edge in $L$ receiving a color in $P$ with equal probability $1/k$, independently of the other edges. We will show using the symmetric local lemma that such a coloring exists which creates no 2-colored $C_4$, and hence, $g(K_{n,n},C_4)\leq n/2+n^{1-\delta}$. 

To this end, we define three types of ``bad events," without which every $C_4$ in our coloring will contain some color exactly once. First, for each pair of adjacent edges $e,f\in E(L)$ and each color $i\in P$, let $A_{e,f,i}$ be the event that $e$ and $f$ receive the same color $i$. Note that $\prob[A_{e,f,i}]=k^{-2}$. Next, for each 4-cycle $D$ in $L$, let $B_D$ be the event that $D$ is properly-colored with two colors from $P$. Then $\prob[B_D]\leq 2k^{-2}$. Finally, for each color $i\in P$ and each 4-cycle $D=xyx'y'$ in $K_{n,n}$ with $xy, x'y'\in E(L)$ and $xy', x'y\in E(F)$ where $xy'$ and $x'y$ received the same color in the first coloring, let $C_{D,i}$ be the event that $xy$ and $x'y'$ both receive color $i$. Then $\prob[C_{D,i}]=k^{-2}$.  Let $\EE$ be the collection of all bad events of the types $A_{e,f,i}$, $B_D$, and $C_{D,i}$. 

Fix an event $E\in \EE$. Note that $E$ is mutually independent from the set of all events $E'$ which do not involve any of the same edges as $E$ in $L$. We call such events edge-disjoint from $E$. There are at most $8\Delta(L)k\leq 8k^2n^{-\delta}$ events of the form $A_{e,f,i}$ which are not edge-disjoint from $E$. In addition, there are at most $4(\Delta(L))^2\leq k^2n^{-\delta}$ events $B_D$ and, by property (4) of Theorem~\ref{thm:coloringproperties}, at most $4(4n^{1-2\delta})k\leq 16k^2n^{-\delta}$ events $C_{D,i}$ which are not edge-disjoint from $E$. Thus, in total, $E$ is mutually independent of all but at most $25k^2n^{-\delta}$ events in $\EE$. By the local lemma, since $25k^2n^{-\delta}\cdot 2k^{-2}\leq 1/4$, there is a coloring of $L$ which contains none of the events in $\EE$, as desired. 
\qed

\bibliography{main}
\bibliographystyle{plain}

\end{document}